\documentclass[pdflatex,sn-mathphys-num]{sn-jnl}%

\usepackage{graphicx}%
\usepackage{multirow}%
\usepackage{amsmath,amssymb,amsfonts}%
\usepackage{amsthm}%
\usepackage{mathrsfs}%
\usepackage[title]{appendix}%
\usepackage{xcolor}%
\usepackage{textcomp}%
\usepackage{manyfoot}%
\usepackage{booktabs}%
\usepackage{algorithm}%
\usepackage{algorithmicx}%
\usepackage{algpseudocode}%
\usepackage{listings}%

\newtheorem{theorem}{Theorem}[section]
\newtheorem{lemma}[theorem]{Lemma}
\newtheorem{corollary}[theorem]{Corollary}
\newtheorem{proposition}[theorem]{Proposition}

\theoremstyle{definition}
\newtheorem{definition}[theorem]{Definition}
\newtheorem{example}[theorem]{Example}

\newtheorem{question}[theorem]{Question}

\theoremstyle{remark}
\newtheorem{remark}[theorem]{Remark}
\begin{document}

\title[Article Title]{The $\phi$-conjugation of quaternionic matrices and generalized Autonne-Takagi factorization}

\author[1]{\fnm{Cailing} \sur{Yao}}\email{yaocl@ccucm.edu.cn}

\author[2]{\fnm{Bingzhe} \sur{Hou}}\email{houbz@jlu.edu.cn}

\author*[3]{\fnm{Yue} \sur{Xin}}\email{2024104@hlju.edu.cn}

\affil[1]{\orgdiv{School of Medical Informatics}, \orgname{Changchun University of Chinese Medicine}, \orgaddress{\street{Jingyue Street}, \city{Changchun}, \postcode{130117}, \country{China}}}

\affil[2]{\orgdiv{School of Mathematics}, \orgname{Jilin University}, \orgaddress{\street{Qianjin Street}, \city{Changchun}, \postcode{130012}, \country{China}}}

\affil[3]{\orgdiv{School of Mathematics and Statistics}, \orgname{Heilongjiang University}, \orgaddress{\street{Xuefu Road}, \city{Harbin}, \postcode{050016}, \country{China}}}


\abstract{Let $\phi$ be a quaternion of modulus $1$. In this article, we study some topics related to $\phi$-conjugation for quaternionic matrices, including  $\phi$-Hermitian matrices, $\phi$-conjugate normal matrices, unitary $\phi$-congruence and  $\phi$-HSH decomposition (decomposition of a $\phi$-Hermitian matrix and a skew $\phi$-Hermitian matrix). In particular, we generalize the Autonne-Takagi factorization of quaternion $\phi$-Hermitian matrices for all unit quaternion $\phi$. This answers the problem proposed by  R. Horn and F. Zhang in the paper ``A generalization of the complex Autonne-Takagi factorization to quaternion matrices, Linear Multilinear A. 60: 1239--1244, 2012''. More precisely, we prove that for every unit quaternion $\phi$, any $\phi$-Hermitian quaternionic matrix can be transformed into a real diagonal matrix under unitary $\phi$-congruence. It should be noted that, when $\phi\in\mathbb{S}$, the resulting real diagonal
matrix can be chosen to have only nonnegative entries.
For $\phi\notin\mathbb{S}$, however, such a choice is
not always possible.
}

\keywords{Quaternionic matrices, $\phi$-Hermitian, Autonne-Takagi factorization, unitary $\phi$-congruence, $\phi$-HSH decomposition}

\pacs[MSC Classification]{15A20, 15A23, 15B33, 15B57}

\maketitle

\section{Introduction and preliminaries}\label{sec1}
Quaternions have extensive applications in many fields, such as physics, computer graphics, robotics, aerospace, virtual reality, and so on.
 The skew field $\mathbb{H}$ of quaternions was discovered by W. Hamilton in 1843. We define $a\in\mathbb{H}$ to be $a=a_0+a_1\boldsymbol{i}+a_2\boldsymbol{j}+a_3\boldsymbol{k},$ where $a_0, a_1, a_2, a_3\in\mathbb{R}$ and $\boldsymbol{i}^2=\boldsymbol{j}^2=\boldsymbol{k}^2=\boldsymbol{ijk}=-1$. Denote the real and imaginary parts of $a$  by $\mathrm{Re}(a)=a_0$ and $\mathrm{Im}(a)=a_1\boldsymbol{i}+a_2\boldsymbol{j}+a_3\boldsymbol{k}$, respectively, the modulus of $a$ by
$|a|=\sqrt{a_0^2+a_1^2+a_2^2+a_3^2}$, and the conjugation of $a$ by $\overline{a}=a_0-a_1\boldsymbol{i}-a_2\boldsymbol{j}-a_3\boldsymbol{k}$. We say $a\in\mathbb{H}$ is invertible if there exists a $b\in\mathbb{H}$ such that $ab=ba=1$ and call $b$ the inverse of $a$.

 Denote by $\mathbb{C}$ the field of complex numbers and $M_{n}(\mathbb{C})$ the set of all $n\times n$ complex matrices, for more details about $M_{n}(\mathbb{C})$, refer to the book \cite{FZZ}. Denote by $M_{n}(\mathbb{H})$ the set of all $n\times n$ quaternionic matrices. If $Q=(q_{ij}) \in M_{n}(\mathbb{H})$, we use $\overline{Q}=(\overline{q_{ij}}),Q^{T}=(q_{ji}),Q^{*}=\overline{Q^{T}}=(\overline{Q})^{T}$ to denote the conjugation, transpose and conjugate transpose of $Q$ respectively. If $Q=\overline{Q}$, then $Q$ is a real matrix. If $Q=Q^{T}$, we say $Q$ is symmetric and if $Q=Q^{*}(Q=-Q^{*})$, then we say $Q$ is Hermitian(skew-Hermitian). For two matrices $A,B\in M_{n}(\mathbb{H})$, if there exists an invertible matrix $P\in  M_{n}(\mathbb{H})$ such that $PAP^{-1}=B$, then we say $A$ is similar to $B$ and denote by $A\sim B$. Especially, when the invertible matrix turns to a unitary matrix $U$, we say $A$ is unitarily equivalent to $B$ and denote by $A\sim_{u} B$. We use $\mathbb{H}^{n}$ to denote the set of all n-dimensional quaternion column vectors. If $\boldsymbol{u}=(u_{1},u_{2},\cdots,u_{n})^{T}, \boldsymbol{v}=(v_{1},v_{2},\cdots,v_{n})^{T}$ are two column vectors in $\mathbb{H}^{n}$, then the inner product of $\boldsymbol{u}$ and $\boldsymbol{v}$ is defined to be
 $$
 \langle \boldsymbol{u},\boldsymbol{v} \rangle=\overline{v_{1}}u_{1}+\cdots+\overline{v_{n}}u_{n}=\boldsymbol{v}^{*}\boldsymbol{u}.
 $$
If $\langle \boldsymbol{u},\boldsymbol{v} \rangle=0$, then we say the two vectors are orthogonal to each other. There have been many classical works related to quaternionic matrices, such as the eigenvalues of quaternionic matrices \cite{FWZ}, \cite{ZFZS}, quaternion matrix equations \cite{HTX} and so on.

The singular value decomposition (SVD) of complex matrices was originally given by L. Autonne \cite{A1915} in $1915$. He also obtained
special singular value decompositions for symmetric, coninvolutory, normal, orthogonal and Lorentzian matrices. The result about the singular value decomposition for a complex symmetric matrix $A$ is that every complex symmetric matrix $A$ can be factored as $A=V\Sigma V^{T}$, where $V$ is complex unitary and $\Sigma$ is nonnegative diagonal \cite{TT}. This decomposition is called the Autonne-Takagi decomposition. The Autonne-Takagi decomposition has wide applications in fields such as quantum
information and signal processing \cite{AT}. The importance of SVD lies in its unification of several core concepts in linear algebra (norms, rank, condition number, subspaces), and its theoretical provision of optimal low-rank approximations. Its applications span multiple fields, including data science, signal processing, control theory, and quantum physics. Whenever problems involve matrix analysis, dimensionality reduction, denoising, compression, or computing pseudoinverses, SVD is almost always one of the most reliable tools.

As a generalization of L. Autonne's result, C. Took, D. Mandic and F. Zhang \cite{TMZ} proposed a unitary diagonalization of a special class of quaternionic matrices, the so-called $\eta$-Hermitian matrices, where $\eta\in\{\boldsymbol{i},\boldsymbol{j},\boldsymbol{k}\}$. Later, R. Horn and F. Zhang \cite{HZ} gave the Autonne-Takagi factorization for quaternion $\eta$-Hermitian matrices where $\eta$ is a unit pure imaginary quaternion. In the meanwhile, they proposed a question in \cite{HZ}.

\begin{question} Let $\eta$ be a unit quaternion but not pure imaginary. Does the Autonne-Takagi factorization  still exist for all quaternion $\eta$-Hermitian matrices?
\end{question}

We conduct a study on this problem and obtain complete results. To clarify this issue, we have provided the following definitions.

Assume $\phi\in\mathbb{H}$, that is, $\phi=\phi_{0}+\phi_{1}\boldsymbol{i}+\phi_{2}\boldsymbol{j}+\phi_{3}\boldsymbol{k}$.
Let $\mathbb{S}^{3}=\{ \phi\in\mathbb{H}; |\phi|^{2}=1\}$ and $\mathbb{S}=\{ \phi\in\mathbb{H}; |\phi|^{2}=1, \phi_{0}=0\}$. Then $\mathbb{S}^{3}\setminus \mathbb{S}=\{ \phi\in \mathbb{S}^{3}; \phi_{0}\neq 0\}$. 
If $\phi=\pm1$, then $\overline{A}^{\phi}=A$ and
$A^{*\phi}=A^*$ for all $A\in M_n(\mathbb{H})$.
In this case, the required real diagonalization follows
from the spectral theorem for quaternionic Hermitian
matrices. Thus, in what follows, we assume that
$\phi\ne\pm1$. For convenience, we set
\[
\mathbb{S}^{\prime}
=
\mathbb{S}^{3}\setminus
\bigl(\mathbb{S}\cup\{1,-1\}\bigr).
\]
\begin{definition}
Let $\phi\in \mathbb{S}^{3}$ and $A\in M_{n}(\mathbb{H})$. Define $\overline{A}^{\phi}=\overline{\phi} A\phi$ to be the $\phi$-conjugation of $A$, and $A^{*{\phi}}=\overline{A^*}^{\phi}(=\overline{\phi }A^*\phi)$ to be the $\phi$-transpose of $A$.
\end{definition}

T. Ell and S. Sangwine studied quaternion involutions and anti-involutions in \cite{TAESJS}. M. Bekar and Y. Yay{\i} conducted a research on the dual quaternion involutions and anti-involutions in \cite{MBYY}. As a matter of fact, when $\phi\in\mathbb{S}$, our definition of $\phi$-conjugation is precisely a quaternion anti-involution and our definition of $\phi$-transpose is a quaternion involution as defined in \cite{TAESJS}.

\begin{definition}
Let $\phi\in \mathbb{S}^{3}$, $A$ and $B$ be  matrices in $M_n(\mathbb{H})$. If there exists an invertible matrix $P$ in $M_n(\mathbb{H})$ such that $PA(\overline{P}^{\phi})^{-1}=B$, then we say $A$ is $\phi$-similar to $B$ and denote by $A\sim_{\phi} B$. In particular, we say $A$ is unitarily $\phi$-congruent to $B$ if there exists a unitary matrix $U$ such that $UA(\overline{U}^{\phi})^{-1}=UAU^{*\phi}=B$ and denote by $A\sim_{\phi-u} B$.
\end{definition}

There is a close connection between unitary $\phi$-congruence and unitary equivalence.
\begin{lemma}\label{264141}
Let $A,B\in M_{n}(\mathbb{H})$, $\phi\in \mathbb{S}^{3}$. Then $A\sim_{\phi-u} B$ if and only if $A\overline{\phi}\sim_{u} B\overline{\phi}$.
\end{lemma}
In particular, if  $\phi\in \mathbb{S}$, then $A\overline{\phi}\sim_{u} B\overline{\phi}$ is equivalent to $\phi A \sim_{u}\phi B$.

Since similar quaternionic matrices share the same right eigenvalues and the right eigenvalues are similar invariants. Likewise, we define the right $\phi$-eigenvalues, and the right $\phi$-eigenvalues defined in such a way are $\phi$-similar invariants.

\begin{definition}
Let $\phi\in \mathbb{S}^{3}$ and $A\in M_n(\mathbb{H})$. If $A\overline{\boldsymbol{v}}^{\phi}=\boldsymbol{v}\lambda$ for some quaternion $\lambda$ and nonzero $\boldsymbol{v}\in \mathbb{H}^{n}$, then we say $\lambda$ is a right $\phi$-eigenvalue of $A$ and $\boldsymbol{v}$ is a right $\phi$-eigenvector of $A$ with respect to $\lambda$. If $A\overline{\boldsymbol{v}}^{\phi}=\lambda\boldsymbol{v}$ for some quaternion $\lambda$ and nonzero $\boldsymbol{v}\in \mathbb{H}^{n}$, then we say $\lambda$ is a left $\phi$-eigenvalue of $A$ and $\boldsymbol{v}$ is a left $\phi$-eigenvector of $A$ with respect to $\lambda$.
\end{definition}
From above definitions we can see that
\begin{enumerate}
\item The right $\phi$-eigenvalues are $\phi$-similar invariants while the left $\phi$-eigenvalues are not.
\item If $A$ and $B$ are $\phi$-similar, then they have the same right $\phi$-eigenvalues.
\item Assume $\lambda$ and $\mu$ are right $\phi$-eigenvalues of $A$. If $\lambda$ and $\mu$ are not $\phi$-similar, then the right $\phi$-eigenvectors with respect to $\lambda$ and $\mu$ are right linearly independent.
\end{enumerate}
  Readers who are interested may further explore the properties of $\phi$-similarity and left or right $\phi$-eigenvalues further, in comparison with the classic similarity and left or right eigenvalues. We will not elaborate too much in the present paper.

A well-known truth is that every normal quaternionic matrix $A$ can be unitarily equivalent to a complex diagonal matrix, and the normality is preserved by unitary equivalence. In particular, if $A$ is Hermitian, then the diagonal matrix is real. If $A$ is skew-Hermitian, then each of the diagonal entry is pure imaginary. However, the normality is not preserved by unitary $\phi$-congruence. We propose a new definition named $\phi$-conjugate normality which is preserved under unitary $\phi$-congruence.

\begin{definition}
Let $\phi\in \mathbb{S}$ and $A\in M_n(\mathbb{H})$. If $A$ satisfies
$$
\overline{A^*A}^{\phi}=AA^*,
$$
then we say $A$ is $\phi$-conjugate normal.
\end{definition}
The reason we assume $\phi\in \mathbb{S}$ is that the two equalities $\overline{A^*A}^{\phi}=AA^*$ and $\overline{AA^*}^{\phi}=A^*A$ do not imply each other when $\phi\in \mathbb{S}^{\prime}$. This asymmetric relationship will cause the definition to appear somewhat unreasonable.
\begin{definition}
Let $\phi\in \mathbb{S}^{3}$. If $A=A^{*{\phi}}(A=-A^{*{\phi}})$, then we say $A$ is $\phi$-Hermitian(skew $\phi$-Hermitian).
\end{definition}

\begin{remark}
Generally speaking, when $\phi\in \mathbb{S}^{\prime}$, the equation $\overline{\overline{A}^{\phi}}^{\phi}=A$ does not hold for most $A\in M_{n}(\mathbb{H})$. Fortunately,  the equation $A=A^{*{\phi}}(A=-A^{*{\phi}})$ is always equivalent to the equation $\overline{A}^{\phi}=A^{*}(\overline{A}^{\phi}=-A^{*})$. This fact will be explained in detail in Proposition \ref{26491}. Therefore, the definition of $\phi$-Hermitian(skew $\phi$-Hermitian) applies perfectly to all $\phi\in \mathbb{S}^{3}$. Especially, when $\phi\in \mathbb{S}$, each $\phi$-Hermitian matrix is $\phi$-conjugate normal.
\end{remark}

The present article is organized as follows. In Section $2$, we review some known results and investigate some new results for the unitary equivalence of quaternionic matrices, which play an important role in the subsequent discussion of unitary $\phi$-congruence.
In Section $3$, we mainly generalize the Autonne-Takagi factorization of quaternion $\phi$-Hermitian matrices for all unit quaternion $\phi$. The discussion is divided into two cases, $\textrm{Re}\phi=0$ and $\textrm{Re}\phi\neq0$.
In Section $4$, we study the unitary $\phi$-congruence for quaternionic matrices. In particular, we introduce $\phi$-HSH decomposition (decomposition of a $\phi$-Hermitian matrix and a skew $\phi$-Hermitian matrix) and then use it to provide necessary and sufficient conditions for unitary $\phi$-congruence.

\section{Unitary equivalence}
Complex unitary matrices preserve inner products, norms, and angles, making them essential for maintaining geometric structure in complex vector spaces. They have many applications such as in quantum mechanics \cite{LW2024}, signal processing \cite{VM1997}, communications \cite{HM2000} and so on. Yu. A$\text{l}^\prime$pin and Kh. Ikramov \cite{Kh2003} reduced the problem of verifying the unitary equivalence between
two complex matrices $A,B\in M_{n}(\mathbb{C})$ to verifying the similarity between two pairs of complex matrices.

\begin{proposition}[\cite{Kh2003}]\label{li}
For any two complex matrices $A, B$ in $M_{n}(\mathbb{C})$, the following assertions are equivalent.
\begin{enumerate}
\item [(1)] A and B are unitarily equivalent;
\item [(2)] the families $\{A, A^{\ast}\}$ and $\{B,B^{\ast}\}$ are unitarily equivalent;
\item [(3)] the families $\{A, A^{\ast}\}$ and $\{B,B^{\ast}\}$ are similar.
\end{enumerate}
\end{proposition}

We have obtained a conclusion similar to Proposition \ref{li} in finite dimensional quaternionic matrix situation.

\begin{theorem}\label{shui}
Let $A,B \in M_{n}(\mathbb{H})$ be any two quaternionic matrices. Then the following statements are equivalent.
\begin{enumerate}
\item [(1)] A and B are unitarily equivalent;
\item [(2)] the families $\{A, A^{\ast}\}$ and $\{B,B^{\ast}\}$ are unitarily equivalent;
\item [(3)] the families $\{A, A^{\ast}\}$ and $\{B,B^{\ast}\}$ are similar.
\end{enumerate}
\end{theorem}
\begin{proof}
First of all, we concentrate on the proof of $(3)\Rightarrow (1)$. Let $P$ be the invertible quaternionic matrix such that
$$
P^{-1}AP=B,~~~~P^{-1}A^{\ast}P=B^{\ast}.
$$
This implies
\begin{equation}\label{1}
PP^{\ast}A=APP^{\ast}.
\end{equation}
By the polar decomposition of a quaternionic matrix (Theorem $7.1$ of \cite{Z}), for any $P\in M_{n}(\mathbb{H})$, there exists a quaternion unitary matrix $U$ and a quaternion positive semidefinite Hermitian matrix $H$ such that $P = HU$, where $H^2=PP^{\ast}$.

Next we aim to prove the equality $HA=AH$. Recall that for a matrix $T$ in $M_{n}(\mathbb{H})$, $T$ can be naturally decomposed into the form $T=T_{1}+\boldsymbol{j}T_{2}$, where $T_{1},T_{2}\in M_{n}(\mathbb{C})$. Then, the complex representation matrix of $T$ is defined to be
$$
T_{\mathbb{C}}=\begin{pmatrix}
T_{1} & -\overline{T_{2}} \\
T_{2} & \overline{T_{1}}
\end{pmatrix}.
$$
Notice the truth that $T$ is a positive semidefinite Hermitian quaternionic matrix in $M_{n}(\mathbb{H})$ if and only if $T_{\mathbb{C}}$ is a positive semidefinite Hermitian complex matrix in
$M_{2n}(\mathbb{C})$. Equation (\ref{1}) is $H^{2}A=AH^{2}$ actually and this is equivalent to $(H^{2})_{\mathbb{C}}A_{\mathbb{C}}=A_{\mathbb{C}}(H^{2})_{\mathbb{C}}$, or equivalently, $(H_{\mathbb{C}})^{2}A_{\mathbb{C}}=A_{\mathbb{C}}(H_{\mathbb{C}})^{2}$.

Now, for the Hermitian complex matrix $(H_{\mathbb{C}})^{2}$, according to the relevant conclusions of the continuous functional calculus for normal operators \cite{JBC}, we have $H_{\mathbb{C}}A_{\mathbb{C}}=A_{\mathbb{C}}H_{\mathbb{C}}$. Therefore, $HA=AH$. Hence
$$
B=P^{-1}AP=U^{-1}H^{-1}AHU=U^{\ast}H^{-1}HAU=U^{\ast}AU.
$$
Since it is not difficult to verify $(1)\Rightarrow (2)$ and $(2)\Rightarrow (3)$, we get the desired result finally.
\end{proof}

It has been proved in \cite{GJSW87} that a complex matrix $A\in M_{n}(\mathbb{C})$ is normal if and only if there exists a polynomial $p$ such that $A^*=p(A)$. We obtain the analogous conclusion for matrices in $M_{n}(\mathbb{H})$. First, we need the following lemma.

\begin{lemma}\label{1011}
Let $\lambda_1,\cdots,\lambda_n\in\mathbb{C}$. Then there exists a real-coefficient polynomial $p$ of degree at most $2n-1$ such that 
$$
p(\lambda_i)=\overline{\lambda_i},\ \ \ 1\leq i\leq n.
$$
\end{lemma}
\begin{proof}
We choose a subset $W$ of $\{\lambda_1,\ldots,\lambda_n\}$ containing
exactly one representative from each class under the equivalence relation
\[
    z\sim w
    \quad\Longleftrightarrow\quad
    z=w\ \text{or}\ z=\overline{w}.
\]
Thus, for any two distinct elements $q_s,q_t\in W$, we have
$q_s\neq q_t$ and $q_s\neq\overline{q_t}$; moreover, for every
$\lambda_i$, there exists some $q_s\in W$ such that
$\lambda_i=q_s$ or $\lambda_i=\overline{q_s}$. We may write
\[
    W=\{q_1,\ldots,q_k\},\qquad 1\leq k\leq n.
\]

We define a new set $W_1$ based on $W$ and make an agreement as follows,

\begin{enumerate}
\item if $q_i\in W$ is nonreal, then $q_i$ and $\overline{q_i}$ are both in $W_1$,
\item if $q_j\in W$ is real, then $q_j$ itself is in $W_1$.
\end{enumerate}
That is, the real elements in $W$ appear only once in $W_1$ and the complex numbers in $W$ appear in conjugate pairs in $W_1$.
 We assign new symbols to the elements in the set $W_1$ and still denote it as set $W_1$. Let $W_1=\{ c_1,\overline{c_1},\cdots,c_l,\overline{c_l}, r_1,\cdots,r_h \}$, then the number of elements in $W_1$ is equal to $2l+h$, which is at most $2n$.
Define the vandermonde matrix $D$ as follows,
$$
D=\begin{pmatrix}
1 & c_1 & \cdots &c_{1}^{2l+h-1}\\
1 & \overline{c_1} & \cdots &\overline{c_{1}}^{2l+h-1}\\
\cdots&\cdots&\cdots&\cdots\\
1 & c_l & \cdots &c_{l}^{2l+h-1} \\
1 & \overline{c_l} & \cdots &\overline{c_{l}}^{2l+h-1} \\
1& r_1 & \cdots &r_{1}^{2l+h-1}\\
\cdots&\cdots&\cdots&\cdots\\
1& r_h & \cdots &r_{h}^{2l+h-1}
\end{pmatrix}.
$$
It can be known that $D$ is invertible by the definition of $W_1$. Then there exists a complex coefficient polynomial $p$ of degree  $2l+h-1$ such that
$$
p(x)=\overline{x}=a_{2l+h-1}x^{2l+h-1}+a_{2l+h-2}x^{2l+h-2}+\cdots+a_{1}x+a_{0}
$$
 for each $x\in W_1$. That is,
 $$
 D\begin{pmatrix}
a_0 \\
a_1\\
\cdots\\
\cdots\\
\cdots\\
\cdots\\
\cdots\\
a_{2l+h-1}
\end{pmatrix}=\begin{pmatrix}
\overline{c_1} \\
c_1\\
\cdots\\
\overline{c_l}\\
c_l\\
r_1\\
\cdots\\
r_h
\end{pmatrix}.
 $$
Let
$$
M=\begin{pmatrix}
\frac{1}{\boldsymbol{i}} & -\frac{1}{\boldsymbol{i}}\\
1 & 1
\end{pmatrix}, \ T=\begin{pmatrix}
M &  &       & & \\
  &M &       & & \\
  &  &\ddots & & \\
  &  &       &M& \\
  &  &       & &I_{h}
\end{pmatrix},
$$
it can be seen both $M$ and $T$ are invertible. Therefore, we have
$$
 TD\begin{pmatrix}
a_0 \\
a_1\\
\cdots\\
\cdots\\
\cdots\\
\cdots\\
\cdots\\
a_{2l+h-1}
\end{pmatrix}=T\begin{pmatrix}
\overline{c_1} \\
c_1\\
\cdots\\
\overline{c_l}\\
c_l\\
r_1\\
\cdots\\
r_h
\end{pmatrix}.
 $$
That is,

$$
\small\begin{pmatrix}
2 & c_1+\overline{c_1} & \cdots &c_{1}^{2l+h-1}+\overline{c_{1}}^{2l+h-1}\\
0 & \frac{1}{\boldsymbol{i}}(c_1-\overline{c_1}) & \cdots &\frac{1}{\boldsymbol{i}}(c_{1}^{2l+h-1}-\overline{c_{1}}^{2l+h-1})\\
\cdots&\cdots&\cdots&\cdots\\
2 & c_l+\overline{c_l} & \cdots &c_{l}^{2l+h-1}+\overline{c_{l}}^{2l+h-1} \\
0 & \frac{1}{\boldsymbol{i}}(c_l-\overline{c_l}) & \cdots &\frac{1}{\boldsymbol{i}}(c_{l}^{2l+h-1}-\overline{c_{l}}^{2l+h-1}) \\
1& r_1 & \cdots &r_{1}^{2l+h-1}\\
\cdots&\cdots&\cdots&\cdots\\
1& r_h & \cdots &r_{h}^{2l+h-1}
\end{pmatrix}\begin{pmatrix}
a_0 \\
a_1\\
\cdots\\
\cdots\\
\cdots\\
\cdots\\
\cdots\\
a_{2l+h-1}
\end{pmatrix}=\begin{pmatrix}
c_1+\overline{c_1} \\
\frac{1}{\boldsymbol{i}}(c_1-\overline{c_1})\\
\cdots\\
c_l+\overline{c_l}\\
\frac{1}{\boldsymbol{i}}(c_l-\overline{c_l})\\
r_1\\
\cdots\\
r_h
\end{pmatrix}\small.
$$
Note that all the elements in the matrix on the left side of the equation and the column vector on the right side of the equation are all real numbers. Therefore, all the coefficients of $p$ are also real numbers, which confirms our conclusion.
\end{proof}
\begin{proposition}\label{2026.1.5.3}
Let $A\in M_{n}(\mathbb{H})$. Then $A$ is normal if and only if there exists a polynomial of degree no more than $2n-1$ with real coefficients such that $A^*=p(A)$.
\end{proposition}
\begin{proof}
We prove the necessity first.

If $A$ is normal, then there exists a unitary matrix $U$ such that
$$
U^{*}AU=\mathrm{diag}(\lambda_1,\cdots,\lambda_n),
$$
where $\lambda_i\in \mathbb{C}$ are right eigenvalues of $A$. By Lemma \ref{1011}, there is a real coefficient polynomial $p$ of degree at most $2n-1$ such that $p(\lambda_i)=\overline{\lambda_i}$. Then
$$
\begin{aligned}
A^*&=U \mathrm{diag}(\overline{\lambda_1},\cdots,\overline{\lambda_n})U^*\\
&=U \mathrm{diag}(p(\lambda_1),\cdots,p(\lambda_n))U^*\\
&=Up(\mathrm{diag}(\lambda_1,\cdots,\lambda_n))U^*\\
&=p(U \mathrm{diag}(\lambda_1,\cdots,\lambda_n)U^*)\\
&=p(A).
\end{aligned}
$$
Now we prove the sufficiency. If there exists a polynomial $p$ of degree at most $2n-1$ with real coefficients such that $A^*=p(A)$, then
$$
A^{*}A=p(A)A=Ap(A)=AA^*,
$$
which implies $A$ is normal.
\end{proof}
\begin{corollary}\label{2026.1.5.2}
If two normal matrices $A,B\in M_{n}(\mathbb{H})$ are similar, then they are unitarily equivalent.
\end{corollary}
\begin{proof}
If $A$ and $B$ are similar, then the right eigenvalues of $A$ and $B$ are the same. Assume complex numbers $\lambda_1,\cdots,\lambda_n$ are representative elements of all the right eigenvalues of $A$ and $B$. Then there exists an invertible matrix $Q$ such that
\begin{equation}\label{2026.1.5.4}
Q^{-1}AQ=B.
\end{equation}
By Lemma~2.3, there exists a polynomial $p$ with real
coefficients and degree at most $2n-1$ such that
\[
p(\lambda_i)=\overline{\lambda_i},
\qquad i=1,\ldots,n.
\]

Since $A$ and $B$ are both normal, the proof of Proposition \ref{2026.1.5.3} yields 
$$
A^*=p(A),\ B^*=p(B).
$$
 Therefore,
 \begin{equation}\label{2026.1.5.5}
 Q^{-1}A^{*}Q=Q^{-1}p(A)Q=p(Q^{-1}AQ)=p(B)=B^{*}.
 \end{equation}
(\ref{2026.1.5.4}) and (\ref{2026.1.5.5}) imply that the families $\{A, A^{*}\}$ and $\{B, B^{*}\}$ are similar, then by Theorem \ref{shui}, $A$ and $B$ are unitarily equivalent.
\end{proof}
For $A,B \in M_{n}(\mathbb{H})$, let
\begin{equation}\label{2026.1.5.6}
A=H_1+H_2,\ B=G_1+G_2
\end{equation}
 where $H_1=\frac{1}{2}(A+A^*)$, $H_2=\frac{1}{2}(A-A^*)$, $G_1=\frac{1}{2}(B+B^*)$ and $G_2=\frac{1}{2}(B-B^*)$. It is obvious that $H_1$ and $G_1$ are Hermitian, $H_2$ and $G_2$ are skew-Hermitian. In fact, $H_1$ and $G_1$, $H_2$ and $G_2$ are all special kinds of normal matrices. Based on this decomposition method for quaternionic matrices, we obtain an alternative equivalent characterization of unitary equivalence for any two  matrices
 $A,B\in M_{n}(\mathbb{H})$.
\begin{theorem}\label{2026.2.5.10}
For $A,B \in M_{n}(\mathbb{H})$, assume $A$ and $B$ have the decomposition as (\ref{2026.1.5.6}). Then the following statements are equivalent.
\begin{enumerate}
\item [(1)] A and B are unitarily equivalent.

\item [(2)] The families $\{ H_1,H_{2}\}$ and $\{ G_1,G_2\}$ are unitarily equivalent.

\item [(3)] The families $\{ H_1,H_2\}$ and $\{ G_1,G_2\}$ are similar.

\end{enumerate}
\end{theorem}
\begin{proof}
$(3)\Rightarrow (2)$. Assume the invertible matrix $P$ satisfies
\begin{equation}\label{2026.1.5.7}
P^{-1}H_{1}P=G_{1},\ P^{-1}H_{2}P=G_{2}.
\end{equation}
Since $H_1$ and $G_1$ are Hermitian, while $H_2$ and $G_2$
are skew-Hermitian, we have
\[
H_1^*=H_1,\qquad H_2^*=-H_2,\qquad
G_1^*=G_1,\qquad G_2^*=-G_2.
\]
Therefore, (\ref{2026.1.5.7}) directly implies
$$
P^{-1}H_1^*P=G_1^*,\qquad
P^{-1}H_2^*P=G_2^*. 
$$
Moreover,
\[
P^{-1}AP
 =P^{-1}(H_1+H_2)P
 =G_1+G_2
 =B,
\]
and
\[
P^{-1}A^*P
 =P^{-1}(H_1-H_2)P
 =G_1-G_2
 =B^*.
\]
Thus, the families $\{A,A^*\}$ and $\{B,B^*\}$ are similar.
By Theorem \ref{shui}, there exists a unitary matrix $U$ such that
\[
U^*AU=B.
\]
Consequently, $U^*A^*U=B^*$, and hence
\begin{equation}\label{2026.1.5.9}
\begin{split}
  U^*H_1U &=\frac12 U^*(A+A^*)U
 =\frac12(B+B^*)
 =G_1,  \\
 U^*H_2U
 &=\frac12 U^*(A-A^*)U
 =\frac12(B-B^*)
 =G_2. 
\end{split}
\end{equation}

Therefore, the families $\{H_1,H_2\}$ and $\{G_1,G_2\}$
are unitarily equivalent.

$(2)\Rightarrow (1)$. Assume $U$ is the unitary matrix that satisfies (\ref{2026.1.5.9}), then
$$
U^{*}AU=U^{*}(H_{1}+H_{2})U=G_{1}+G_{2}=B.
$$
That is, $A$ is unitarily equivalent to $B$.

$(1)\Rightarrow (3)$. The proof of this direction is obvious.
\end{proof}

\section{Generalized Autonne-Takagi factorization for $\phi$-Hermitian matrices}

In this section, we mainly focus on providing a detailed exposition of the diagonalization of quaternion $\phi$-Hermitian matrices under unitary $\phi$-congruence transformations, where $\phi$ is a unit quaternion. Prior to that, we first show some basic properties that may be used later.

\begin{lemma}
Assume $\phi\in \mathbb{S}^{3}$,  $A, B\in M_n(\mathbb{H})$. Then the following statements hold.
\begin{enumerate}

    \item  $(\overline{A}^{\phi})^* = \overline{A^*}^{\phi}=A^{*{\phi}}$, $(A^{*{\phi}})^{*}=(A^*)^{*{\phi}}$.

    \item If $A$ is invertible, then $(\overline{A}^{\phi})^{-1}=\overline{A^{-1}}^{\phi}$, $(A^{*{\phi}})^{-1}=(A^{-1})^{*{\phi}}$.

    \item $\overline{A\boldsymbol{x}}^{\phi}=\overline{A}^{\phi}\overline{\boldsymbol{x}}^{\phi}$ for every $\boldsymbol{x}\in\mathbb{H}^n$.
    \item $\overline{A+B}^{\phi}=\overline{A}^{\phi}+\overline{B}^{\phi}$, $(A+B)^{*{\phi}}=A^{*{\phi}}+B^{*{\phi}}$.
    \item $\overline{AB}^{\phi}=\overline{A}^{\phi}\cdot\overline{B}^{\phi}$,  $(AB)^{*{\phi}}= B^{*{\phi}}A^{*{\phi}}$.

\item $\overline{A^T}^{\phi}=(\overline{A}^{\phi})^T=\overline{A^{*{\phi}}}=(\overline{A})^{*{\phi}}$,

\item $(A^{*{\phi}})^{T}=(A^{T})^{*{\phi}}=\overline{\overline{A}^{\phi}}=\overline{\overline{A}}^{\phi}  $,
\end{enumerate}
\end{lemma}

\begin{lemma}[\cite{Z}]\label{1151}
If $a$ and $b$ are two nonzero quaternions, then $a\sim b$ if and only if $|a|=|b|$ and $Re(a)=Re(b)$.
\end{lemma}

\begin{lemma}\label{1154}
Let $\phi\in\mathbb{S}^{3}$. If $a$ and $b$ are two nonzero quaternions, then $a\sim_{\phi} b$ if and only if $|a|=|b|$ and $Re(a\overline{\phi})=Re(b\overline{\phi})$, or equivalently, $Re(\phi a)=Re(\phi b)$.
\end{lemma}
\begin{proof}
Assume $a\sim_{\phi} b$, then there exists a nonzero quaternion $p$ such that $pa(\overline{p}^{\phi})^{-1}=b$. This is equivalent to $p(a\overline{\phi})p^{-1}=b\overline{\phi}$, which implies the conclusion by Lemma \ref{1151}. The other direction can be proved by Lemma \ref{1151}.
\end{proof}

Equations of the type $ax+b=xc$ have been considered in \cite{RE1944}. It is shown that
if $a, b, c$ are quaternions and  $a$ is not similar to $c$, then $ax+b=xc$ has a
 solution in $\mathbb{H}$.

\begin{lemma}\label{815}
Let $\phi\in\mathbb{S}^{3}$, let a, b, c be quaternions and  a is not $\phi$-similar to c. Then the equation $ax+b=\overline{x}^{\phi}c$ has a solution in $\mathbb{H}$.
\end{lemma}
\begin{proof}
Notice the truth that $a$ and $c$ are not $\phi$-similar is equivalent to $\phi a$ and $\phi c$ are not similar, since $a\sim_{\phi}c$ if and only if $\phi a\sim\phi c$. Assume the equation
$$
(\phi a)x+d=x(\phi c)
$$
has a solution $x_{0}$ for any quaternion $d$. Let $d=\phi b$, then we obtain $ax_{0}+b=\overline{x_0}^{\phi}c$.
\end{proof}

For a matrix $A\in M_{n}(\mathbb{H})$, if $A$ is in triangular form (upper or lower), then every diagonal element is a right eigenvalue of $A$. In fact, if $A$ is in triangular form, then every diagonal element is a right $\phi$-eigenvalue of $A$ for all $\phi\in\mathbb{S}^{3}$. The proof is absolutely analogous to the process of right eigenvalues of a diagonal quaternionic matrix in \cite{JL51} and the conclusion of Lemma \ref{815} may be put to use.

\begin{theorem}\label{26416}
Let $A\in M_n(\mathbb{H})$, $\phi\in\mathbb{S}^{3}$. If $A$ is of triangular form, then all the right $\phi$-eigenvalues of $A$ are exactly all the diagonal elements of $A$ together with all the quaternions $\phi$-similar to the diagonal elements.
\end{theorem}

\begin{lemma}[\cite{Z}]\label{2026.1.13.1}
Let $A\in M_{n}(\mathbb{H})$. Then there exists a unitary matrix $U\in M_{n}(\mathbb{H})$ such that $UAU^{*}$ is an upper triangular matrix with diagonal entries $h_{1}+k_{1}\boldsymbol{i},\cdots,h_{n}+k_{n}\boldsymbol{i}$ are right eigenvalues of $A$ and $k_{t}\geq 0$ for $1\leq t\leq n$.
\end{lemma}
\begin{corollary}[\cite{Z}]
Let $A\in M_{n}(\mathbb{H})$. Then $A$ is normal if and only if there exists a unitary matrix $U\in M_{n}(\mathbb{H})$ such that
$$
UAU^{*}=\mathrm{diag}(h_{1}+k_{1}\boldsymbol{i},\cdots,h_{n}+k_{n}\boldsymbol{i})
$$
and $A$ is Hermitian if and only if $k_{1}=\cdots=k_{n}=0$.
\end{corollary}
The analogous conclusions can be obtained by Lemma \ref{264141}.
\begin{corollary}
Let $A\in M_n(\mathbb{H})$, $\phi\in\mathbb{S}^{3}$. Then there exists a unitary matrix $U\in M_n(\mathbb{H})$ such that $UAU^{*\phi}$ is an upper triangular matrix.
\end{corollary}

\begin{corollary}\label{11101}
Let $A\in M_n(\mathbb{H})$, $\phi\in\mathbb{S}$. Then $A$ is $\phi$-conjugate normal if and only if $A$ is unitarily $\phi$-congruent to a diagonal matrix.
\end{corollary}
\begin{remark}\label{264142}
As well known, every normal quaternionic matrix can be unitarily equivalent to a complex diagonal matrix. From Corollary \ref{11101}, we assume $A$ is $\phi$-conjugate normal, then $\phi A$ is normal. Let $U$ be the unitary matrix such that
\begin{equation}\label{2026.1.7.1}
U(\phi A)U^{*}=\mathrm{diag}(\phi\lambda_1,\cdots,\phi\lambda_n),
\end{equation}
where the $\phi\lambda_s$ is a complex number with non negative imaginary part for each $1\leq s\leq n$. (\ref{2026.1.7.1}) is equivalent to
$$
\overline{U}^{\phi}AU^{*}=\overline{\phi}U(\phi A)U^{*}=\mathrm{diag}(\lambda_1,\cdots,\lambda_n).
$$
Therefore, every $\phi$-conjugate normal quaternionic matrix $A$ can be unitarily $\phi$-congruent to a diagonal matrix $T_{A}=\mathrm{diag}(\lambda_1,\cdots,\lambda_n)$ where the $\lambda_{s}$ satisfies $\phi\lambda_{s}$ is a complex number with non negative imaginary part for each $1\leq s\leq n$. We say $T_{A}$ is the canonical form of $A$ under unitary $\phi$-congruence. This canonical form is unique up to a permutation
of its diagonal entries.
\end{remark}
At the end of this section, we will answer the question R. Horn and F. Zhang asked in \cite{HZ}. We will discuss in two parts according to whether the real part of $\phi$ is zero. First, we discuss the case that $\phi\in \mathbb{S}^{\prime}$. In this case, $\overline{\phi}=\phi^{-1}$.

Now we introduce and explain some notations. For $\phi\in\mathbb{H}$, define $\{\phi\}^{\prime}=\{ q\in\mathbb{H}; q\phi=\phi q\}$. That is,  $\{\phi\}^{\prime}$ is the commutant of $\phi$. For $\phi\in\mathbb{H}\setminus\mathbb{R}$, define $\phi=\phi_{0}+r_{\phi}I_{\phi}$, where $I_{\phi}$ is the imaginary unit of $\phi$ and $r_{\phi}\in \mathbb{R}$ is the coefficient of $I_{\phi}$. Denote by $\mathbb{C}_{I_{\phi}}=\{ a+bI_{\phi};a,b\in\mathbb{R}\}$. Then $\{\phi\}^{\prime}=\mathbb{C}_{I_{\phi}}$ for $\phi\in\mathbb{H}\setminus \mathbb{R}$.

In fact, for a fixed $\phi$, $\mathbb{C}_{I_{\phi}}$ is isomorphic to $\mathbb{C}$. We only need to establish a one-to-one correspondence between $a+bI_{\phi}$ and $a+b\boldsymbol{i}$. We use the notation $M_{n}(\mathbb{C}_{I_{\phi}})$ to denote the set of $n\times n$ quaternionic matrices whose entries belong to $\mathbb{C}_{I_{\phi}}$. That is
$$
M_{n}(\mathbb{C}_{I_{\phi}})=\{ A=(a_{ij})\in M_{n}(\mathbb{H});a_{ij}\in \mathbb{C}_{I_{\phi}} for ~i,j=1,\cdots,n\}.
$$
Essentially, $M_{n}(\mathbb{C}_{I_{\phi}})\cong M_{n}(\mathbb{C})$ and
\begin{enumerate}
\item [1.]$B\in M_{n}(\mathbb{C}_{I_{\phi}})$ if and only if $B\in M_{n}(\mathbb{C}_{I_{\overline{\phi}}})$,
\item [2.]$B\in M_{n}(\mathbb{C}_{I_{\phi}})$ if and only if $B^{\ast}\in M_{n}(\mathbb{C}_{I_{\phi}})$.
\end{enumerate}
\begin{lemma}\label{26414}
Assume a quaternion $\psi\in \mathbb{H}\setminus\mathbb{R}$. If a quaternion $\nu$ satisfies
$$
\psi\nu=\nu\psi,
$$
then $\nu\in \mathbb{C}_{I_{\psi}}$.
\end{lemma}
\begin{lemma}\label{26413}
Let $\psi$ be a quaternion with a non zero real part. If a quaternion $a$ satisfies
$$
\overline{\psi}a\psi=\overline{a},
$$
then $a\in\mathbb{R}$.
\end{lemma}
\begin{proof}
Let $\psi=\psi_{0}+\psi_{1}\boldsymbol{i}+\psi_{2}\boldsymbol{j}+\psi_{3}\boldsymbol{k}$ and $a=a_{0}+a_{1}\boldsymbol{i}+a_{2}\boldsymbol{j}+a_{3}\boldsymbol{k}$, where $\psi_{0}\neq 0$. 
If $a=0$, then $a\in\mathbb{R}$. If $a\neq 0$, then we have $|\psi|=1$. If $\psi\in\mathbb{R}$, then $\psi=\pm 1$. The conclusion holds evidently. If $\psi\notin\mathbb{R}$, then $\psi\in \mathbb{S}^{\prime}$.

Notice that the equality $\overline{\psi}a\psi=\overline{a}$ is equivalent to $a\psi=\psi\overline{a}$. From this, we obtain the following system of equations
\begin{displaymath}
\left \{ \begin{array}{ll}
a_{1}\psi_{1}+a_{2}\psi_{2}+a_{3}\psi_{3}=0\\
a_{1}\psi_{0}=0\\
a_{2}\psi_{0}=0\\
a_{3}\psi_{0}=0
\end{array} \right.
\end{displaymath}
Since $\psi_{0}\neq 0$, it follows that $a_{1}=a_{2}=a_{3}=0$, which means $a\in\mathbb{R}$.
\end{proof}
In fact, for $\phi\in \mathbb{S}^{\prime}$, $M_{n}(\mathbb{C}_{I_{\phi}})$ and $M_{n}(\mathbb{C}_{I_{\phi^{2}}})$ are exactly the same set. Thus, for the sake of convenience, in what follows we shall predominantly use the symbol $M_{n}(\mathbb{C}_{I_{\phi}})$ in place of $M_{n}(\mathbb{C}_{I_{\phi^{2}}})$ when $\phi\in \mathbb{S}^{\prime}$.
\begin{proposition}\label{26491}
Let $\phi\in \mathbb{S}^{\prime}$, $A\in M_{n}(\mathbb{H})$. If $A$ is $\phi$-Hermitian, then $A\in M_{n}(\mathbb{C}_{I_{\phi}})$.
\end{proposition}
\begin{proof}
Since $A$ is $\phi$-Hermitian, we have ${\overline{A}^{\phi}}=A^{\ast}$. By comparing $\overline{A}^{\phi}$ and $A^{\ast}$ term by term, we obtain

\begin{equation}\label{26411}
\overline{\phi}a_{ii}\phi=\overline{a_{ii}}
\end{equation}
 for $1\leq i\leq n$ and
\begin{equation}\label{26412}
\overline{\phi}a_{ij}\phi=\overline{a_{ji}},\ \ \overline{\phi}a_{ji}\phi=\overline{a_{ij}}
\end{equation}
for all $1\leq i\neq j\leq n$.

From (\ref{26411}) we know that $a_{ii}\in\mathbb{R}$ for $1\leq i\leq n$ by Lemma \ref{26413}, and from (\ref{26412}) we obtain $a_{ij}\phi^{2}=\phi^{2}a_{ij}$ for all $1\leq i\neq j\leq n$. Therefore, $a_{ij}\in \mathbb{C}_{I_{\phi^{2}}}$ for $1\leq i,j\leq n$ by Lemma \ref{26414}, which implies $A\in M_{n}(\mathbb{C}_{I_{\phi^{2}}})$, or equivalently, $A\in M_{n}(\mathbb{C}_{I_{\phi}})$.
\end{proof}

Our main theorem in this part is as follows.

\begin{theorem}\label{26417}
Let $\phi\in \mathbb{S}^{\prime}$, $A\in M_{n}(\mathbb{H})$. If $A$ is $\phi$-Hermitian, then $A$ is Hermitian and there exists a unitary matrix $U\in M_{n}(\mathbb{C}_{I_{\phi}})$ and a real diagonal matrix $\Delta$ such that $\Delta=UAU^{\ast\phi}$.
\end{theorem}
\begin{proof}
Suppose that $A$ is a $\phi$-Hermitian matrix, by the definition of $\phi$-Hermitian matrix, we have ${\overline{A}^{\phi}}=A^{\ast}$. That is, $\overline{\phi}A\phi=A^{\ast}$. By Proposition \ref{26491}, $A\in M_{n}(\mathbb{C}_{I_{\phi}})$ and consequently, $A\phi=\phi A$. Therefore,
\begin{equation*}\label{h}
A^{\ast}=\overline{\phi} A \phi=\overline{\phi}\phi A=A.
\end{equation*}
This conclusion implies that for all $\phi\in \mathbb{S}^{\prime}$, a $\phi$-Hermitian matrix is a Hermitian matrix in $M_{n}(\mathbb{C}_{I_{\phi}})$.

Notice that $M_{n}(\mathbb{C}_{I_{\phi}})$ is isomorphic to $M_{n}(\mathbb{C})$ for any $\phi\in \mathbb{S}^{\prime}$. Then, it follows from $A$ being a Hermitian matrix in $M_{n}(\mathbb{C}_{I_{\phi}})$ that there exists a unitary matrix $U\in M_{n}(\mathbb{C}_{I_{\phi}})$
such that $UAU^{\ast}=\Delta$, where $\Delta$ is real diagonal. Furthermore, by $U\in M_{n}(\mathbb{C}_{I_{\phi}})$, we have
$$
\Delta=UAU^{\ast}=UAU^{\ast}\overline{\phi}\phi=UA\overline{\phi}U^{\ast}\phi=UAU^{\ast\phi}.
$$
This completes the proof of the theorem.
\end{proof}

\begin{remark}
Note that for $\phi\in\mathbb{S}$, the real diagonal matrix $\Delta$ can be further taken to be nonnegative. This result can be referred to \cite{HZ}, while we can not generally guarantee that all elements of $\Delta$ are nonnegative when $\phi\in \mathbb{S}^{\prime}$. This is essentially because

\begin{enumerate}
\item For a non zero $\sigma\in \mathbb{R}$, $\sigma$ and $-\sigma$ are $\phi$-similar when $\phi\in\mathbb{S}$ by Lemma \ref{1154}. However, it is not true for $\phi\in \mathbb{S}^{\prime}$.
\item The unitary $\phi$-congruence keeps the right $\phi$-eigenvalues invariant and Theorem \ref{26416} indicates the location of the right $\phi$-eigenvalues for the matrix with triangular form.
\end{enumerate}\end{remark}
For a more intuitive understanding of this fact, we provide a concrete example. 
 The following example shows that, when $\mathbb{S}^{3}\setminus \mathbb{S}$, a
 $\phi$-Hermitian matrix cannot always be transformed into a
 nonnegative real diagonal matrix under unitary $\phi$-congruence.

\begin{example}
	Let $\phi\in \mathbb{S}^{3}\setminus \mathbb{S}$ and
	\[
	A=\begin{pmatrix}
		-1&0\\
		0&-2
	\end{pmatrix}.
	\]
	Then $A^{*\phi}=A$, so $A$ is $\phi$-Hermitian.
	We show that any real diagonal matrix obtained from $A$
	under unitary $\phi$-congruence has a negative diagonal entry.
	
	Suppose that
	\[
	D=UAU^{*\phi}=\operatorname{diag}(d_1,d_2),
	\qquad d_1,d_2\in\mathbb{R},
	\]
	for some unitary matrix $U\in M_2(\mathbb{H})$.
	By the definition of the $\phi$-transpose, we have
	\[
	D(\phi I_2)=U\bigl(A(\phi I_2)\bigr)U^*.
	\]
	Recall that
	\[
	\operatorname{Re}\operatorname{tr}(XY)
	=
	\operatorname{Re}\operatorname{tr}(YX)
	\]
	for any two square quaternionic matrices $X$ and $Y$
	of the same order. Hence
	\[
	\begin{aligned}
		(d_1+d_2)\operatorname{Re}(\phi)
		&=\operatorname{Re}\operatorname{tr}\bigl(D(\phi I_2)\bigr)\\
		&=\operatorname{Re}\operatorname{tr}
		\bigl(U A(\phi I_2)U^*\bigr)\\
		&=\operatorname{Re}\operatorname{tr}\bigl(A(\phi I_2)\bigr)\\
		&=-3\operatorname{Re}(\phi).
	\end{aligned}
	\]
	Since $\phi\in S'$, we have $\operatorname{Re}(\phi)\ne0$,
	and consequently $d_1+d_2=-3$.
	Therefore, at least one diagonal entry of $D$ is negative.
	In particular, $A$ cannot be transformed into a nonnegative
	real diagonal matrix under unitary $\phi$-congruence.
\end{example}

Now we proceed to prove the case that $\phi\in\mathbb{S}$. In this case, we have $\phi=\phi_{1}\boldsymbol{i}+\phi_{2}\boldsymbol{j}+\phi_{3}\boldsymbol{k}$ and $\phi_{1}^{2}+\phi_{2}^{2}+\phi_{3}^{2}=1$. Then $\overline{\phi}=\phi^{-1}=-\phi$ and $\phi^{2}=-1$. Therefore the $\phi$-conjugation of $A$ equals to $\overline{A}^{\phi}=-\phi A\phi$, the $\phi$-transpose of $A$ equals to $A^{*{\phi}}=-\phi A^*\phi$ and the following equalities hold
$$
\overline{\overline{A}^{\phi}}^{\phi}=A=(A^{*{\phi}})^{*{\phi}},  \overline{A^{*{\phi}}}^{\phi}=(\overline{A}^{\phi})^{*{\phi}}=A^{*}.
$$

As we have stated, the analogous result of Theorem \ref{26417} in the case $\phi\in\mathbb{S}$ has been proved in \cite{HZ}. In this section, we provide an alternative proof of this conclusion and it appears as a corollary of a theorem. In order to give a complete statement and proof of the main theorem of this part, we first make some preliminary preparations.
\begin{definition}
Let $A\in M_{n}(\mathbb{H})$. For $\omega\in\mathbb{H}$, if there exists a nonzero vector $\boldsymbol{x}\in\mathbb{H}^{n}$ such that $A\boldsymbol{x}=\boldsymbol{x}\omega$, then we call $\omega$ a right eigenvalue of $A$, and call $\boldsymbol{x}$ a right eigenvector of $A$ with respect to  $\omega$. Denote by $\sigma_{r}(A)$ the set of all right eigenvalues of $A$.
\end{definition}
Denote by $\mathcal{H}_{q}$ a separable quaternion Hilbert space and $L_{r}(\mathcal{H}_{q})$  the set of all bounded quaterinionic right linear operators on $\mathcal{H}_{q}$.
Right eigenvalues and right eigenvectors could be naturally defined for bounded quaterinionic right linear operators. Given any $T\in L_{r}(\mathcal{H}_{q})$. For any $\omega\in\mathbb{H}$, define the map $T-I\omega$ by
\[
(T-I\omega)\boldsymbol{y}=T\boldsymbol{y}-\boldsymbol{y}\omega, \ \ \ \ \text{for any} \ \boldsymbol{y}\in\mathcal{H}_{q}.
\]
Then, $\omega$ is a right eigenvalue of $T$ if and only if $\ker(T-I\omega)$ is nonzero. Moreover, $\ker(T-I\omega)$ is the right eigenvector space of $T$ with respect to  $\omega$.
X. Feng, B. Hou and K. Ji \cite{FHJ} studied the right eigenvector space of  $T\in L_{r}(\mathcal{H}_{q})$. They showed that if $\omega\in\sigma_{r}(T)$, then $\ker(T-I\omega)$ is a right linear vector space over the division ring $\{\omega\}^{\prime}$, and  for any $0\neq q\in\mathbb{H}$,
$$
\ker(T-Iq^{-1}\omega q)=\{ \boldsymbol{x}q;\boldsymbol{x}\in \ker(T-I\omega)\}.
$$
Obviously, those results hold for finite dimensional quaternionic matrices.

\begin{lemma}[{\cite[Lemma 3.1]{FHJ}}]
Let $A\in M_{n}(\mathbb{H})$ and $\omega\in\sigma_{r}(A)$. Then, for any $0\neq q\in\mathbb{H}$,
$$
\ker(A-Iq^{-1}\omega q)=\{ \boldsymbol{x}q;\boldsymbol{x}\in \ker(A-I\omega) \}.
$$
\end{lemma}

\begin{theorem}[\cite{Z}]\label{114}
Let $u_1$ be a unit column vector of $n$ quaternion components. Then there exist $n-1$ unit column vectors $u_2,\ \cdots,\ u_n$ of $n$ column
components such that $\{u_1,\ u_2,\ \cdots,\ u_n\}$ is an orthogonal set, i.e., $u_s^*u_t=0,\ s\ne t$.
\end{theorem}
\begin{lemma}\label{113}
Let $\phi\in\mathbb{S}$, $\alpha\in \mathbb{H}$. If $\alpha\overline{\alpha}^{\phi}=\lambda$ for some nonnegative real number $\lambda$, then there exists a unit $p\in \mathbb{H}$ such that $p\alpha(\overline{p}^{\phi})^{-1}=\sqrt{\lambda}$.
\end{lemma}

\begin{proof}
Since $\phi\in S$, we have $\overline{\phi}=-\phi$ and
$\phi^2=-1$. Hence
\[
\alpha\alpha^\phi
 =-\alpha\phi\alpha\phi
 =-(\alpha\phi)^2
 =\lambda.
\]

If $\lambda=0$, then $\alpha=0$. Therefore,
\[
p\alpha(p^\phi)^{-1}=0=\sqrt{\lambda}
\]
for any unit quaternion $p$.

Now assume that $\lambda>0$ and set
\[
\psi=\frac{\alpha\phi}{\sqrt{\lambda}}.
\]
Then
\[
\psi^2
 =\frac{(\alpha\phi)^2}{\lambda}
 =-1,
\]
so $\psi\in S$. Thus $\phi$ and $\psi$ have the same real part and
the same modulus. Consequently, they are similar, and hence there
exists a nonzero quaternion $p$ such that
\[
p\psi p^{-1}=\phi,
\]
or equivalently,
\[
p\psi=\phi p.
\]
After replacing $p$ by $p/|p|$, we may assume that $p$ is a unit
quaternion.

Since $\alpha\phi=\sqrt{\lambda}\psi$ and
$p^\phi\phi=\phi p$, we obtain
\[
\begin{aligned}
p\alpha(p^\phi)^{-1}=\sqrt{\lambda}
&\iff p\alpha=\sqrt{\lambda}\,p^\phi\\
&\iff p\alpha\phi=\sqrt{\lambda}\,\phi p\\
&\iff p\psi=\phi p.
\end{aligned}
\]
The last equality holds by the choice of $p$, which proves the result.
\end{proof}

The following theorem is our main theorem of this part.
\begin{theorem}\label{1142}
Let $\phi\in\mathbb{S}$, $A\in M_n(\mathbb{H})$. Then the following two statements are equivalent.
\begin{enumerate}
\item [(a)] There exists a unitary matrix $U\in M_n(\mathbb{H})$ and an upper triangular matrix $\Delta\in M_n(\mathbb{H})$ such that $A=U\Delta U^{*\phi}$ with all diagonal entries of $\Delta$ are nonnegative.

\item [(b)] All the right eigenvalues of $A\overline{A}^{\phi}$ are nonnegative.
\end{enumerate}

\end{theorem}
\begin{proof}
	(a) $\Rightarrow$ (b).
	Suppose that $A=U\Delta U^{*\phi}$, where $U$ is unitary
	and $\Delta$ is upper triangular with nonnegative real
	diagonal entries $\sigma_1,\ldots,\sigma_n$.
	Since $\phi\in S$, we have
	\[
	A\overline{A}^{\phi}
	=
	U\Delta\overline{\Delta}^{\phi}U^*.
	\]
	The matrix $\Delta\overline{\Delta}^{\phi}$ is upper
	triangular with diagonal entries
	$\sigma_1^2,\ldots,\sigma_n^2$.
	Therefore, all the right eigenvalues of
	$A\overline{A}^{\phi}$ are nonnegative.
	
	(b) $\Rightarrow$ (a).
	We proceed by induction on $n$.
	If $n=1$, the conclusion follows from Lemma \ref{113}.
	Suppose that $n\geq2$ and that the conclusion holds
	for matrices of order $n-1$.
	
	Let $\lambda\geq0$ be a right eigenvalue of
	$A\overline{A}^{\phi}$, and choose a nonzero vector $\boldsymbol{x}$
	such that
	\[
	A\overline{A}^{\phi}\boldsymbol{x}=\boldsymbol{x}\lambda.
	\]
	We first show that there exist a nonzero vector $\boldsymbol{v}$
	and a quaternion $\alpha$ such that
	\[
	A\overline{\boldsymbol{v}}^{\phi}=\boldsymbol{v}\alpha,
	\qquad
	\alpha\overline{\alpha}^{\phi}=\lambda.
	\]
	If $A\overline{\boldsymbol{x}}^{\phi}$ and $\boldsymbol{x}$ are right linearly
	dependent, then
	\[
	A\overline{\boldsymbol{x}}^{\phi}=\boldsymbol{x}\alpha
	\]
	for some $\alpha\in\mathbb{H}$.
	Applying the map $z\mapsto A\overline{z}^{\phi}$ twice
	gives
	\[
	\boldsymbol{x}\lambda
	=
	A\overline{A}^{\phi}\boldsymbol{x}
	=
	\boldsymbol{x}\alpha\overline{\alpha}^{\phi}.
	\]
	Thus we may take $\boldsymbol{v}=\boldsymbol{x}$.
	
	If $A\overline{\boldsymbol{x}}^{\phi}$ and $\boldsymbol{x}$ are right linearly
	independent, set
	\[
	\sigma=\sqrt{\lambda},
	\qquad
	\boldsymbol{v}=A\overline{\boldsymbol{x}}^{\phi}+\boldsymbol{x}\sigma.
	\]
	Then $\boldsymbol{v}\ne0$ and
	\[
	A\overline{\boldsymbol{v}}^{\phi}
	=
	\boldsymbol{x}\lambda+A\overline{\boldsymbol{x}}^{\phi}\sigma
	=
	\boldsymbol{v}\sigma.
	\]
	Thus we may take $\alpha=\sigma$.
	
	After normalizing $\boldsymbol{v}$, we may assume that $\boldsymbol{v}$ is a unit
	vector. By Lemma \ref{113}, there exists a unit quaternion $p$
	such that
	\[
	p\alpha(\overline{p}^{\phi})^{-1}
	=
	\sqrt{\lambda}.
	\]
	Put $\boldsymbol{w}=\boldsymbol{v}p^{-1}$ and $\sigma=\sqrt{\lambda}$.
	Then $\boldsymbol{w}$ is a unit vector and
	\[
	A\overline{\boldsymbol{w}}^{\phi}=\boldsymbol{w}\sigma.
	\]
	Extend $\boldsymbol{w}$ to an orthonormal basis of $\mathbb{H}^n$,
	and let $V$ be the unitary matrix whose columns are
	these basis vectors. We obtain
	\[
	C=V^*A\overline{V}^{\phi}
	=
	\begin{pmatrix}
		\sigma&r\\
		0&A_2
	\end{pmatrix},
	\]
	where $r\in M_{1,n-1}(\mathbb{H})$ and
	$A_2\in M_{n-1}(\mathbb{H})$.
	Moreover,
	\[
	C\overline{C}^{\phi}
	=
	V^*A\overline{A}^{\phi}V
	=
	\begin{pmatrix}
		\sigma^2&
		\sigma\overline{r}^{\phi}
		+r\overline{A_2}^{\phi}\\
		0&A_2\overline{A_2}^{\phi}
	\end{pmatrix}.
	\]
	Consequently, all the right eigenvalues of
	$A_2\overline{A_2}^{\phi}$ are nonnegative.
	
	By the induction hypothesis, there exists a unitary
	matrix $W\in M_{n-1}(\mathbb{H})$ such that
	\[
	A_2=W\Delta_2W^{*\phi},
	\]
	where $\Delta_2$ is upper triangular with nonnegative
	real diagonal entries.
	Let $E=\operatorname{diag}(1,W)$. Then
	\[
	\Delta=E^*C\overline{E}^{\phi}
	=
	\begin{pmatrix}
		\sigma&r\overline{W}^{\phi}\\
		0&\Delta_2
	\end{pmatrix}
	\]
	is upper triangular with nonnegative real diagonal
	entries. Setting $U=VE$, we obtain
	\[
	A=U\Delta U^{*\phi}.
	\]
	This completes the induction.
\end{proof}

\begin{example}
Let $\phi=\boldsymbol{j}$ and $A=
\begin{pmatrix}
1 & 0 \\
\boldsymbol{j} & 1 \\
\end{pmatrix}$,
$\Delta=
\begin{pmatrix}
1 & \boldsymbol{j} \\
0 & 1
\end{pmatrix}$. Then there exists
$U=
\begin{pmatrix}
0 & 1 \\
1 & 0
\end{pmatrix}$  such that $A=U\Delta U^{*\boldsymbol{j}}=U\Delta U^{*\phi}$.
\end{example}

\begin{corollary}\label{UIF1}(Autonne-Takagi factorization)
Let $\phi\in\mathbb{S}$ and $A\in M_n(\mathbb{H})$ be $\phi$-Hermitian. Then there exists a unitary matrix $U\in M_n(\mathbb{H})$ and a real nonnegative diagonal matrix
$\Sigma$  such that $A=U\Sigma U^{*\phi}$.
\end{corollary}

\begin{proof}
	Since $\phi\in S$ and $A$ is $\phi$-Hermitian,
	we have
	\[
	\overline{A}^{\phi}=A^*.
	\]
	Consequently,
	\[
	A\overline{A}^{\phi}=AA^*
	\]
	is positive semidefinite Hermitian, and all its right
	eigenvalues are nonnegative.
	
	By Theorem \ref{1142}, there exist a unitary matrix $U$
	and an upper triangular matrix $\Sigma$ with nonnegative
	real diagonal entries such that
	\[
	A=U\Sigma U^{*\phi}.
	\]
	Taking the $\phi$-transpose and using $\phi\in S$,
	we obtain
	\[
	A^{*\phi}=U\Sigma^{*\phi}U^{*\phi}.
	\]
	Since $A=A^{*\phi}$, it follows that
	\[
	\Sigma=\Sigma^{*\phi}.
	\]
	The matrix $\Sigma$ is upper triangular, whereas
	$\Sigma^{*\phi}$ is lower triangular.
	Therefore, $\Sigma$ is diagonal.
	Its diagonal entries are nonnegative by construction.
	
	Moreover,
	\[
	AA^*=U\Sigma^2U^*.
	\]
	Hence the diagonal entries of $\Sigma$ are the singular
	values of $A$.
\end{proof}

\section{Unitary $\phi$-congruence}
Throughout this section, we assume that $\phi\in S$.
In this part, we show equivalent statements for two $\phi$-Hermitian quaternionic matrices $A$ and $B$ to be unitarily $\phi$-congruent. At the same time, we give the $\phi$-HSH decomposition of quaternionic matrices and characterize the unitarily $\phi$-congruence between two quaternionic matrices $A$ and $B$ from the perspective of $\phi$-HSH decomposition. We begin with some preliminary work so that the main results can be stated more clearly. Among these, some conclusions related to unitary $\phi$-congruence are parallel to those on unitary equivalence in Section $2$.

\begin{lemma}\cite{Kh2006}\label{812}
Let $A, B\in M_{n}(\mathbb{C})$. If A and B are unitarily congruent, then $A_{L}=\overline{A}A$ and $B_{L}=\overline{B}B$ are unitarily equivalent.
\end{lemma}
However, the truth in Lemma \ref{812} does not hold for matrices in $M_{n}(\mathbb{H})$ for the reason that $B=U^{T}AU$ does not mean $\overline{B}=U^{*}\overline{A}~ \overline{U}$ because $\overline{MN}\neq \overline{M} ~\overline{N}$ for two quaternionic matrices $M$ and $N$ in general. Fortunately, the definition  ``unitary $\phi$-congruence''  can make a conclusion analogous to Lemma \ref{812} hold true.

\begin{proposition}\label{p42}
Let $A, B\in M_{n}(\mathbb{H})$. If $A$ and $B$ are unitarily $\phi$-congruent, then $A_{L}=\overline{A}^{\phi}A(A_{R}=A\overline{A}^{\phi})$ and $B_{L}=\overline{B}^{\phi}B(B_{R}=B\overline{B}^{\phi})$ are unitarily equivalent.
\end{proposition}

\begin{lemma}\cite{Kh2006}\label{wanle}
If $A\in M_{n}(\mathbb{C})$ is a conjugate-normal matrix, then $A_{L}=\overline{A}A$ is normal.
\end{lemma}
The $\phi$-conjugate normality of a quaternionic matrix has analogous property to the conjugate-normality of a complex matrix.
\begin{proposition}\label{2026.1.5.1}
If $A\in M_{n}(\mathbb{H})$ is $\phi$-conjugate normal, then $A_{L}=\overline{A}^{\phi}A$ and $A_{R}=A\overline{A}^{\phi}$ are normal.
\end{proposition}
\begin{proof}
If $A$ is $\phi$-conjugate normal, then
$$
\overline{A^*A}^{\phi}=\overline{A^*}^{\phi}\overline{A}^{\phi}=AA^*,
$$

$$
\overline{AA^*}^{\phi}=\overline{A}^{\phi}\overline{A^*}^{\phi}=A^*A.
$$
Thus
$$
A_{L}A_{L}^{*}=\overline{A}^{\phi}AA^{*}\overline{A^*}^{\phi}=\overline{A}^{\phi}\overline{A^*}^{\phi}\overline{A}^{\phi}\overline{A^*}^{\phi}=(\overline{A}^{\phi}\overline{A^*}^{\phi})^{2}=(A^*A)^{2},
$$
$$
A_{L}^{*}A_{L}=A^{*}\overline{A^*}^{\phi}\overline{A}^{\phi}A=A^{*}AA^{*}A=(A^{*}A)^{2}.
$$
That is, $A_{L}A_{L}^{*}=A_{L}^{*}A_{L}$, so $A_L$ is normal. The proof of $A_R$ is analogous.
\end{proof}
\begin{remark}
The normality of $A_L$ and $A_{R}$ does not imply $A$ is $\phi$-conjugate normal. For instance, let
 $$
A_{1}= \begin{pmatrix}
0 & 1 \\
0 & 0 \\
\end{pmatrix}, A_{2}= \begin{pmatrix}
0 & 1 \\
2 & 0 \\
\end{pmatrix}.
  $$
Then $A_{1L}=0$ and $A_{2R}=2I_2$ are normal. But
$$
\{ A_{1}A_{1}^{*}\}_{11}=1\neq 0=\{\overline{A_{1}^*}^{\phi}\overline{A_{1}}^{\phi}\}_{11}$$
 and
$$
\{ A_{2}A_{2}^{*}\}_{11}=1\neq 4=\{\overline{A_{2}^*}^{\phi}\overline{A_{2}}^{\phi}\}_{11},$$
which implies $A_1$ and $A_2$ are not $\phi$-conjugate normal.
\end{remark}

\begin{lemma}[\cite{HZ}]\label{2026.1.13.2}
Let $A\in M_{n}(\mathbb{H})$. If $A=-A^{*}$, that is, $A$ is skew-Hermitian, then there exists a unitary matrix $U\in M_{n}(\mathbb{H})$ such that
$UAU^{*}$ is a diagonal matrix with all the diagonal elements pure imaginary.
\end{lemma}
\begin{remark}\label{2026.1.13.4}
Together with Lemma \ref{2026.1.13.1} and Lemma \ref{2026.1.13.2}, it can be seen that a quaternion skew-Hermitian matrix $A$ can be unitarily equivalent to a diagonal matrix whose diagonal entries are $k_{1}\boldsymbol{i},\cdots,k_{n}\boldsymbol{i}$, where $k_{t}\geq 0$ for $1\leq t\leq n$.
\end{remark}
\begin{lemma}\label{2026.1.13.3}
If $A\in M_{n}(\mathbb{H})$, then $A$ is $\phi$-Hermitian if and only if $\phi A$ is skew-Hermitian.
\end{lemma}

With the above conclusions, we obtain some equivalent descriptions about the unitary $\phi$-congruence of two $\phi$-Hermitian matrices $A,B\in M_{n}(\mathbb{H})$.

\begin{theorem}\label{8181}
Let $\phi\in\mathbb{S}$. Suppose $A$ and $B$ are two  $\phi$-Hermitian matrices in $M_n(\mathbb{H})$. Then the following statements are equivalent.
\begin{enumerate}
\item [(1)] A and B are unitarily $\phi$-congruent.

\item [(2)] $A_{L}=\overline{A}^{\phi}A$ and $B_{L}=\overline{B}^{\phi}B$ are unitarily equivalent.

\item [(3)] $A_{R}=A\overline{A}^{\phi}$ and $B_{R}=B\overline{B}^{\phi}$ are unitarily equivalent.

\item [(4)] The families $\{ A_L, A_{L}^{*}\}$ and $\{ B_L, B_{L}^{*}\}$ are unitarily equivalent.

\item [(5)] The families $\{ A_L, A_{L}^{*}\}$ and $\{ B_L, B_{L}^{*}\}$ are  similar.

\item [(6)] The families $\{ A_R, A_{R}^{*}\}$ and $\{ B_R, B_{R}^{*}\}$ are  unitarily equivalent.

\item [(7)] The families $\{ A_R, A_{R}^{*}\}$ and $\{ B_R, B_{R}^{*}\}$ are  similar.

\end{enumerate}

\end{theorem}
\begin{proof}
	By Corollary \ref{UIF1}, there exist unitary matrices $U,V$
	and nonnegative real diagonal matrices $\Sigma_A,\Sigma_B$
	such that
	\[
	A=U\Sigma_AU^{*\phi},
	\qquad
	B=V\Sigma_BV^{*\phi}.
	\]
	Since $\phi\in S$, it follows that
	\[
	A\overline{A}^{\phi}=U\Sigma_A^2U^*,
	\qquad
	B\overline{B}^{\phi}=V\Sigma_B^2V^*,
	\]
	and
	\[
	\overline{A}^{\phi}A
	=
	\overline{U}^{\phi}\Sigma_A^2U^{*\phi},
	\qquad
	\overline{B}^{\phi}B
	=
	\overline{V}^{\phi}\Sigma_B^2V^{*\phi}.
	\]
	
	(1) $\Rightarrow$ (2) and (3).
	These implications follow from Proposition \ref{p42}.
	
	(2) $\Rightarrow$ (1).
	Suppose that $\overline{A}^{\phi}A$ and
	$\overline{B}^{\phi}B$ are unitarily equivalent.
	Then $\Sigma_A^2$ and $\Sigma_B^2$ are unitarily
	equivalent. Hence their real eigenvalues, counted
	with multiplicities, coincide.
	After arranging the diagonal entries in the same order,
	we have
	\[
	\Sigma_A^2=\Sigma_B^2.
	\]
	Since $\Sigma_A$ and $\Sigma_B$ are nonnegative,
	we obtain $\Sigma_A=\Sigma_B$.
	Thus $A$ and $B$ are unitarily $\phi$-congruent.
	
	(3) $\Rightarrow$ (1).
	The same argument applies to
	$A\overline{A}^{\phi}$ and $B\overline{B}^{\phi}$.
	
	Finally, Theorem \ref{shui} gives the equivalences
	\[
	(2)\Longleftrightarrow(4)\Longleftrightarrow(5)
	\]
	and
	\[
	(3)\Longleftrightarrow(6)\Longleftrightarrow(7).
	\]
	This completes the proof.
\end{proof}

\begin{lemma}
Let $\phi\in S$ and $A\in M_n(\mathbb H)$. If $A$ is
$\phi$-conjugate normal, then there exists a polynomial $Q$
with real coefficients and degree at most $2n-1$ such that
\[
\bigl(A\overline{A}^{\,\phi}\bigr)^*
   =Q\bigl(A\overline{A}^{\,\phi}\bigr).
\]
\end{lemma}

\begin{proof}
	Since $A$ is $\phi$-conjugate normal,
	Proposition \ref{2026.1.5.1} implies that
	$A\overline{A}^{\phi}$ is normal.
	Applying Proposition \ref{2026.1.5.3} to this matrix, we obtain
	a polynomial $Q$ with real coefficients and degree
	at most $2n-1$ such that
	\[
	\bigl(A\overline{A}^{\phi}\bigr)^*
	=
	Q\bigl(A\overline{A}^{\phi}\bigr).
	\]
\end{proof}

\begin{proposition}
Let $\phi\in\mathbb{S}$ and $A,B\in M_{n}(\mathbb{H})$ be two $\phi$-Hermitian matrices. If $A$ and $B$ are $\phi$-similar, then they are unitarily $\phi$-congruent.
\end{proposition}
\begin{proof}
	Since $\phi\in S$ and $A,B$ are $\phi$-Hermitian,
	we have
	\[
	\overline{A}^{\phi}=A^*,
	\qquad
	\overline{B}^{\phi}=B^*.
	\]
	Thus $A\overline{A}^{\phi}=AA^*$ and
	$B\overline{B}^{\phi}=BB^*$ are Hermitian.
	
	Suppose that $A$ and $B$ are $\phi$-similar.
	Then there exists an invertible matrix $P$ such that
	\[
	B=PA(\overline{P}^{\phi})^{-1}.
	\]
	Since $\phi\in S$, applying $\phi$-conjugation gives
	\[
	\overline{B}^{\phi}
	=
	\overline{P}^{\phi}\overline{A}^{\phi}P^{-1}.
	\]
	Therefore,
	\[
	B\overline{B}^{\phi}
	=
	P A\overline{A}^{\phi}P^{-1}.
	\]
	Hence the Hermitian matrices
	$A\overline{A}^{\phi}$ and $B\overline{B}^{\phi}$
	are similar. By Corollary \ref{2026.1.5.2}, they are unitarily
	equivalent. Theorem \ref{8181} now implies that $A$ and $B$
	are unitarily $\phi$-congruent.
\end{proof}

 Now we let $A=S+K$ where $S=\frac{1}{2}(A+\overline{A^*}^{\phi})$ is $\phi$-Hermitian and $K=\frac{1}{2}(A-\overline{A^*}^{\phi})$ is skew $\phi$-Hermitian. In fact, $S$ and $K$ are both special $\phi$-conjugate normal matrices.
\begin{definition}
The decomposition $A=S+K$ as above is called the $\phi$-HSH decomposition of A.
\end{definition}
\begin{theorem}
Assume $A\in M_{n}(\mathbb{H})$ and $\phi\in\mathbb{S}$, then A is $\phi$-conjugate normal if and only if $S\overline{K}^{\phi}=K\overline{S}^{\phi}$. Furthermore, the equation  is preserved by unitary $\phi$-congruence.
\end{theorem}
\begin{proof}
Let $A=S+K$, then
$$
A^*A=S^*S+S^*K+K^*S+K^*K,~~~AA^*=SS^*+SK^*+KS^*+KK^*.
$$
Note S is $\phi$-Hermitian and K is skew $\phi$-Hermitian, hence A is $\phi$-conjugate normal if and only if $S\overline{K}^{\phi}=K\overline{S}^{\phi}$. Now assume $S_1=U^*S\overline{U}^{\phi}$ and $K_1=U^*K\overline{U}^{\phi}$. Then
$$
S_1\overline{K_1}^{\phi}=U^*S\overline{U}^{\phi}\overline{U^*}^{\phi}\overline{K}^{\phi}U=U^*K\overline{S}^{\phi}U=U^*K\overline{U}^{\phi}\overline{U^*}^{\phi}\overline{S}^{\phi}U=K_1\overline{S_1}^{\phi}.
$$
\end{proof}

\begin{theorem}\label{264143}
Let $A=S_{1}+K_{1},B=S_{2}+K_{2}$ be the $\phi$-HSH decomposition of quaternionic matrices $A$ and $B$ in $M_{n}(\mathbb{H})$. Then the following statements are equivalent.
\begin{enumerate}
\item [(1)] A and B are unitarily $\phi$-congruent.

\item [(2)] The families $\{ S_1,K_1\}$ and $\{ S_2,K_2\}$ are unitarily $\phi$-congruent.

\item [(3)] The families $\{ S_1,K_1\}$ and $\{ S_2,K_2\}$ are $\phi$-similar.
\end{enumerate}
\end{theorem}
\begin{proof}
$(3)\Rightarrow (1)$. Assume the pairs $\{ S_1,K_1\}$ and $\{ S_2,K_2\}$ are $\phi$-similar, this is equivalent to the pairs $\{ \phi S_1,\phi K_1\}$ and $\{ \phi S_2,\phi K_2\}$ are similar. And furthermore, this is equivalent to the pairs $\{ \phi (S_1+K_1),\phi(K_1-S_1)\}$ and $\{ \phi(S_2+K_2),\phi(K_2-S_2)\}$ are similar. Notice that

$$
\phi A=\phi(S_1+K_1), (\phi A)^{\ast}=A^{\ast}\overline{\phi}=-(S_{1}^{\ast}+K_1^{\ast})\phi=-(\overline{S_{1}}^{\phi}-\overline{K_{1}}^{\phi})\phi=\phi(K_1-S_1).
$$
Therefore, the pairs $\{\phi A,(\phi A)^{\ast}\}$ and $\{\phi B,(\phi B)^{\ast}\}$ are similar, which is equivalent to $\phi A$ and $\phi B$ are unitarily equivalent by Theorem \ref{shui}. And finally, $A$ and $B$ are unitarily $\phi$-congruent.

The proof of $(1)\Rightarrow (2)$ and $(2)\Rightarrow (3)$ are obvious.
\end{proof}

\section*{Statements and Declarations}

\noindent \textbf{Ethics approval}

\noindent Not applicable.

\noindent \textbf{Competing interests}

\noindent The authors declare that there is no conflict of interest or competing interest.

\noindent \textbf{Authors' contributions}

\noindent All authors contributed equally to this work.

\noindent \textbf{Availability of data and materials}

\noindent Data sharing is not applicable to this article as no data sets were generated or analyzed during the current study.



\begin{thebibliography}{10}
\bibitem{Kh2003}
Yu. A$\text{l}^\prime$pin, Kh. Ikramov. 
\newblock On the unitary similarity of matrix families. 
\newblock {\em Mat. Zametki}, 74(6): 815--826, 2003; translation in Math. Notes 74(5--6): 772--782, 2003.


\bibitem{A1915}
L. Autonne. 
\newblock Sur les matrices hypohermitiennes et sur les matrices unitaires. 
\newblock {\em Ann. Univ. Lyon, Nouvelle Serie I, Fasc}, 38: 1--77, 1915.

\bibitem{MBYY}
M. Bekar, Y. Yay{\i}.
\newblock Dual Quaternion Involutions and Anti--Involutions.
\newblock {\em Adv Appl Clifford Al}, 23: 577--592, 2013.

\bibitem{JL51}
J. Brenner. 
\newblock Matrices of quaternions. 
\newblock {\em Pac. J. Math}, 1(3): 329--335, 1951.

\bibitem{JBC}
J. Conway. 
\newblock A Course in Functional Analysis.
\newblock {\em Grad. Texts in Math, 96. Springer--Verlag, New York}, 1985.

\bibitem{TAESJS}
T. Ell, S. Sangwine.
\newblock Quaternion involutions and anti--involutions.
\newblock {\em Comput. Math. Appl}, 53: 137--143, 2007.

\bibitem{FWZ}
F. Farid, Q. Wang, F. Zhang. 
\newblock On the eigenvalues of quaternion matrices.
\newblock {\em Linear Multilinear A}, 59(4): 451--473, 2011.


\bibitem{FHJ}
X. Feng, B. Hou, K. Ji.
\newblock Cowen--Douglas operators on quaternionic Hilbert spaces.
\newblock {\em RACSAM Rev. R. Acad. A., forthcoming}, 2026.

\bibitem{GJSW87}
R. Grone, C. Johnson, E. Sa, H. Wolkowicz.
\newblock Normal matrices.
\newblock {\em Linear Algebra Appl}, 87: 213--225, 1987.

\bibitem{HTX}
J. Huang, Y. Tan, K. Xu.
\newblock Circulant matrix solution to a quaternion matrix equation and its optimal approximation.
\newblock {\em J. Math}, 34(2): 353--359, 2014.

\bibitem{HM2000}
B. Hochwald, T. Marzetta.
\newblock Unitary space-time modulation for multiple--antenna communications in Rayleigh flat fading.
\newblock {\em IEEE T Inform Theory}, 46(2): 543--564, 2000.

\bibitem{HZ}
R. Horn, F. Zhang.
\newblock A generalization of the complex Autonne--Takagi factorization to quaternion matrices.
\newblock {\em Linear Multilinear A}, 60: 1239--1244, 2012.


\bibitem{Kh2006}
Kh. Ikramov.
\newblock A criterion for the unitary congruence of conjugate--normal matrices.
\newblock {\em Dokl. Ross. Akad. Nauk}, 410(1): 17--18, 2006.

\bibitem{RE1944}
R. Johnson.
\newblock On the equation $\chi = \gamma + \beta$ over an algebraic division ring. 
\newblock {\em Bull. Amer. Math. Soc}, 50: 202--207, 1944.

\bibitem{LW2024}
K. Lai, X. Wang.
\newblock Group sparse matrix optimization for efficient quantum state transformation.
\newblock {\em Phys Rev A}, 110, 13 pp, 2024.

\bibitem{VM1997}
S. Phoong, P. Vaidyanathan.
\newblock Paraunitary Filter Banks Over Finite Fields.
\newblock {\em IEEE T. Signal Proces}, 45(6): 1443--1457, 1997.


\bibitem{TT}
T. Takagi.
\newblock On an algebraic problem related to an analytic theorem of carath\'{e}odory and fej\'{e}r
and on an allied theorem of landau.
\newblock {\em In: Japanese Journal of Mathematics: Transactions and
Abstracts. Vol. 1. Tokyo: The Mathematical Society of Japan}, 83--93, 1924.


\bibitem{AT}
A. Teretenkov.
\newblock Singular value decomposition for skew--Takagi factorization with quantum
applications.
\newblock {\em Linear Multilinear A}, 70(22): 7762--7769, 2022.


\bibitem{TMZ}
C. Took, D. Mandic, F. Zhang.
\newblock On the unitary diagonalization of a special class of quaternion matrices.
\newblock {\em Appl Math Lett}, 24: 1806--1809, 2011.

\bibitem{ZFZS}
F. Zhang.
\newblock Ger\v{s}gorin type theorems for quaternionic matrices.
\newblock {\em Linear Algebra Appl}, 424(1): 139--153, 2007.

\bibitem{FZZ}
F. Zhang.
\newblock Matrix Theory: Basic Results and Techniques.
\newblock {\em second ed., Springer, New York}, 2011.

\bibitem{Z}
F. Zhang.
\newblock Quaternions and matrices of quaternions.
\newblock {\em Linear Algebra Appl},
251: 21--57, 1997.
\end{thebibliography}
\end{document}